\documentclass[12pt,reqno]{amsart}
\usepackage[headings]{fullpage}
\usepackage{amssymb,amsmath,mathtools,bbm,tikz,tikz-cd,color,relsize,multirow}
\usepackage[all,cmtip]{xy}
\usepackage{url}
\usepackage{soul}
\usepackage[group-separator={,},group-minimum-digits={3}]{siunitx}
\usepackage{subcaption}

\usepackage{parskip}
\usetikzlibrary{positioning}
\usetikzlibrary{calc}
\usetikzlibrary{decorations.markings}
\usetikzlibrary{arrows}
\usetikzlibrary{calc}
\usetikzlibrary{decorations.markings}
\tikzstyle{hvector}=[inner sep=2pt,draw=blue!50,fill=blue!10,thick]
\tikzstyle{unit}=[inner sep=2pt,shape=circle, draw]
\tikzstyle{counit}=[inner sep=2pt,shape=circle, draw,fill=gray]
\tikzstyle{antipode}=[inner sep=2pt,shape=rectangle, draw]
\tikzstyle{cocycle}=[inner sep=2pt,shape=circle, draw]
\tikzstyle{twistedm}=[inner sep=2pt,shape=circle, fill=gray]
\tikzstyle{autom}=[inner sep=2pt,shape=circle, draw]
\tikzstyle{coact}=[inner sep=2pt,shape=circle, fill=black]

\usepackage[bookmarks=true,%
    colorlinks=true,%
    linkcolor=blue,%
    citecolor=blue,%
    filecolor=blue,%
    menucolor=blue,%
    urlcolor=blue,%
    breaklinks=true]{hyperref}
\usepackage{slashed}    
\usepackage{verbatim}
\usepackage[normalem]{ulem}  
\usepackage{diagbox}
\usepackage[shortlabels]{enumitem}

\usepackage{algorithm}
\usepackage{chngcntr}
\counterwithin{algorithm}{section}
\usepackage[noend]{algpseudocode}

\usepackage{xspace}
\algnewcommand\True{\textbf{True}\xspace}
\algnewcommand\False{\textbf{False}\xspace}

\usepackage{cleveref}

\usepackage{tikz,tikz-cd}
\usetikzlibrary{knots, hobby, decorations.pathreplacing,
  shapes.geometric, calc}

\usepackage{orcidlink}

\newtheorem{theorem}{Theorem}[section]
\theoremstyle{definition}
\newtheorem{proposition}[theorem]{Proposition}
\newtheorem{lemma}[theorem]{Lemma}
\newtheorem{definition}[theorem]{Definition}

\newtheorem{conjecture}[theorem]{Conjecture}
\newtheorem{question}[theorem]{Question}

\def\BZ{\mathbbm Z}

\def\coeff{\mathrm{coeff}}

\def\be{\begin{equation}}
\def\ee{\end{equation}}

\def\LG{\mathrm{LG}}

\def\sl21{\mathfrak{sl}_{2|1}}

\def\sdeghV2{\mathrm{sdeg}_\hbar^{V_2}}

\DeclareMathOperator{\GL}{GL}

\DeclareMathOperator{\conv}{conv}

\DeclareMathOperator{\dom}{dom}

\DeclareMathOperator{\face}{face}

\DeclareMathOperator{\supp}{supp}

\newcommand{\N}{\mathbb{N}}
\newcommand{\Q}{\mathbb{Q}}

\newcommand{\R}{\mathbb{R}}

\newcommand{\Z}{\mathbb{Z}}

\begin{document}

\title[Positivity and log concavity of the Links--Gould polynomial of knots]{
Positivity and log concavity of the Links--Gould polynomial of knots}

\author{Stavros Garoufalidis}
\address[Stavros Garoufalidis]{
  International Center for Mathematics, Department of Mathematics \\
  Southern University of Science and Technology \\
  Shenzhen, China \newline
  {\tt \url{http://people.mpim-bonn.mpg.de/stavros}} }
\email{stavros@mpim-bonn.mpg.de}

\author{Shana Yunsheng Li}
\address[Shana Yunsheng Li]{
  Department of Mathematics \\
  University of Illinois \\
  Urbana, IL, USA \newline
  {\tt \url{https://shana-y-li.github.io}}}
\email{yl202@illinois.edu}

\author{Josephine Yu}
\address[Josephine Yu]{School of Mathematics \\ 
Georgia Institute of Technology \\ Atlanta GA, USA \newline 
\tt\url{https://sites.gatech.edu/josephineyu/}}
\email{jyu@math.gatech.edu}


\date{29 May 2026}

\begin{abstract}
  Motivated by the recent work of Harper--Kohli--Song--Tahar,  
  we formulate a positivity, hole-free, and log-concavity conjecture for the 
  Links--Gould polynomial of alternating links and verify it for all 51.3 million
  alternating knots with at most 19 crossings. All but 544 of those knots
  satisfy a stronger type-B log-concavity condition characterized by the slopes of
  edges in the subdivision of the monomial support induced by the log coefficients.
\end{abstract}

\maketitle

{\footnotesize
\tableofcontents
}


\section{Introduction}
\label{sec.intro}

It is well-known that every Laurent polynomial $f(t) \in \BZ[t^{\pm 1}]$ that
satisfies $f(1)=1$ can be realized as the Alexander polynomial of (infinitely many)
knots in $S^3$. This classical result was proven by Levine~\cite{Levine} using a 
Seifert matrix construction. For alternating knots, however, the situation is 
different. On one hand, there is no characterization of the set of Alexander polynomial
of alternating knots or links. On the other hand, Fox observed in the sixties that
the absolute values of the coefficients of the Alexander polynomial form a unimodal
sequence~\cite{Fox}, this being known as the Fox Trapezoidal Conjecture. An 
additional, also conjectural, hole-free and log-convexity property was observed by
Stoimenow~\cite{Stoimenow}. 

Our paper, however, is not about the long established conjectures about the
classical Alexander polynomial of alternating links, but about a quantum relative
of it, namely the Links--Gould polynomial associated with the 4-dimensional irreducible
representations of the Lie superalgebra $\mathfrak{sl}(2|1)$~\cite{Links-Gould}.
The Links--Gould polynomial of a link $L$ in $S^3$ is a 2-variable polynomial
$\LG_L(t_1,t_2) \in \BZ[t_1^{\pm 1},t_2^{\pm 1}]$ which is related to the Alexander
polynomial $\Delta_L(t)$ in two different ways \cite{DIL:LG,Ben-Michael:LG-Alex, 
KP:LG-Alex}
\be
\label{eq.2special}
\LG_L(t,t^{-1})= \Delta_L(t)^2, \qquad  \LG_L(t, -t^{-1}) =\Delta_L(t^2) \,.
\ee
The relation of the Links--Gould polynomial with the Alexander polynomial and 
with the the $V_1$-polynomial of~\cite{GK:multi} (which will play a role in the 
confirmation of our conjecture below) is discussed in~\cite{Ga:skein}.

Motivated by the properties of the Alexander polynomial of an alternating knot
and the above mentioned relation with the Links--Gould polynomial, the authors
in~\cite{HKST} formulated several
conjectures regarding the support and the coefficients of the Links--Gould
2-variable polynomials, and tested them for all alternating
knots with at most 12 crossings and a few further knots with at most 16 crossings. 
The conjectures themselves have several interconnected versions making it somewhat 
difficult to decide on the importance of each conjecture. 

In this note we formulate a single stronger conjecture concerning the coefficients
of the Links--Gould polynomials of alternating knots, and we confirm it for
all alternating knots with at most 19 crossings. 
To phrase our conjecture, we need some standard terminology from geometric
combinatorics.  Basic concepts can be found, for example, in~\cite{Ziegler,
triangulations}.

A {\em point configuration} is a finite subset  $A$ of the lattice $\BZ^d$.  
We say that $A$ is 
{\em hole-free} if
\be
\conv(A) \cap \BZ^d = A,
\ee
where $\conv(A)$ denotes the convex hull of $A$ in $\R^d$.
We say that a function $h: A \to \R$ is {\em concave extendable} if it can be 
extended to a concave function $\conv(A)\rightarrow \R$; equivalently, for every
relation of the form $b = c_1 a_1+ \dots +c_k a_k$, for
$a_1,\dots,a_k, b \in A$, $c_i>0$, and $\sum_{i=1}^k c_i = 1$, we have
\be
\label{hconv}
h(b) \geq c_1 h(a_1)+\dots+c_k h(a_k).
\ee
We say that $g : A \rightarrow \R_{>0}$ is {\em log-concave extendable} (or just 
{\em log-concave} for short) if the function $\log g : A \rightarrow \R$ given 
by $a \mapsto \log(g(a))$ is concave extendable.  This notion of log-concavity 
is strictly stronger than the ``2d-log-concavity'' used in~\cite{HKST}.

In this paper we consider functions $h: A \rightarrow \R$ arising from 
a Laurent polynomial
\be
p(x) = \sum_{a \in A} h(a) x^a \in \R[x^{\pm 1}]
\ee
where we denote $a=(a_1,\dots,a_d) \in \BZ$, $x=(x_1,\dots,x_d)$, $x^a=x_1^{a_1} \dots
x_d^{a_d}$ and $h(a) \in \R$. 
A Laurent polynomial $p(x)$
gives rise to a point configuration called the {\em support} of $p(x)$: \[\supp(p) := 
\{a \in \Z^d \mid \text{ the monomial } x^a \text{ has nonzero coefficient in } 
p(x)\},\] 
and the coefficient 
function $\coeff_p: \supp(p) \to \R\setminus \{0\}$ which sends a $a\in\supp(p)$ to the 
corresponding coefficient $h(a)$ of $x^a$ in $p(x)$.  

\begin{conjecture}
\label{conj.1}
For every alternating knot $K$, the polynomial $\LG_K(-t_1,-t_2)$ has positive
coefficients and hole-free support, and the coefficient function 
$\coeff_{\LG_K(-t_1,-t_2)}$ is log-concave.
\end{conjecture}

A numerical confirmation of the above conjecture requires data of the Links--Gould
polynomial, and these are obtained using the explicit relation~\cite{Ga:skein} 
\be
\LG_L(t_1,t_2) = V_1(t_1^{1/2}t_2^{-1/2},t_1^{-1/2}t_2^{-1/2})
\ee
between the Links--Gould polynomial and the $V_1$-polynomial of~\cite{GK:multi}
(where $V_{1,L}(t,q) \in \BZ[q^{\pm 1/2},t^{\pm 1}]$), together with 
the values of the $V_1$-polynomial ~\cite{GL:VnData} computed in~\cite{GL:patterns}.

In the course of confirming the above conjecture for knots with at most 19 
crossings, we found that for most
alternating knots the coefficient function of $\LG(-t_1,-t_2)$ satisfies a 
stronger {\em log concavity of type-B} defined in Section~\ref{sec.typeB}, which 
not only is strictly stronger than concave-extendability, but also easier to check. 

\begin{theorem}
\label{thm.1}
Conjecture~\ref{conj.1} holds for all \num{51280976} alternating knots with at 
most 19 crossings. All but 544 of them satisfy the stronger type-B log-concavity
defined in Section~\ref{sec.typeB}.
\end{theorem}

The number of non-type-B alternating knots with at most 19 crossings are shown 
in the following table.
\begin{table}[htpb!]
\begin{center}
\begin{tabular}{|r|r|r|r|r|r|r|r|r|}
\hline
crossings & $\leq 12$ & 13 & 14 & 15 & 16 & 17 & 18 & 19
\\ \hline
\# alt-knots & \num{1851} & \num{4878} & \num{19536} & \num{85263} & \num{379799} 
& \num{1769979} & \num{8400285} &\num{40619385}
\\ 
\# non-type-B & 0 & 1 & 0 & 7 & 15 & 46 & 142 & 333
\\ \hline
\end{tabular}
\end{center}
\label{tab:knots}
\end{table}

The eight non-type-B knots with $\leq 15$ crossings are
\be
13a4593, \,\,
15a57208, \,\,
15a57492, \,\,
15a79064, \,\,
15a79080, \,\,
15a82231, \,\,
15a84755, \,\,
15a84831.
\ee
The list of all 544 non-type-B knots with $\leq 19$ crossings is given in 
Appendix~\ref{append.nontypeB}.

The proof of \Cref{thm.1} is by computation, and the plan of the computation is
as follows.
Since the vast majority of knots satisfy the stronger (and easier to check) 
type-B concavity, we discuss type-B concavity first in Section~\ref{sec.typeB}. Next 
we give an algorithm for certifying log-concavity in general, which uses exact
arithmetic without having to compute logarithms, given that our $\LG$-coefficients 
take integer values.  We also discuss how to certify computation of subdivisions 
of the support of  $\LG$ induced by the log coefficients. All our code and data 
are available at \cite{GLY:LGData}.


\section{Computations, examples and patterns}
\label{sec.patterns}

In the course of verifying \Cref{thm.1} we discovered further patterns
of the Links--Gould polynomial of alternating knots which we want to share as
soon as possible postponing the precise definitions, proofs
and algorithms to a later section.

\subsection{Slopes of the edges in subdivision}
\label{sec:slopes}

For any Laurent polynomial
$\sum_{a \in A} h(a) x^a$ in $\R_+[x^{\pm 1}]
$ with support $A \subset \Z^n$, the log coefficients induce a subdivision of $A$, 
as the projection of the upper hull of the set $\{(a,h(a)): a \in A\}$ in $\R^{n+1}$.
The function $h$ being log-concave of type-B implies that this subdivision has edges
only in directions of the type-B root system, $\pm e_i$ and $e_i \pm e_j$.  In other
words, the edges have slopes $0, \pm 1$, and $\infty$.
See Figure~\ref{fig.11a100} for examples of type-B subdivisions.  As stated in
Theorem~\ref{thm.1}, 99.999\% of the alternating knots with up to 19 crossings 
exhibit only the type-B slopes.

\begin{figure}[htpb]
  \centering
  \begin{tikzpicture}[
    scale=0.6,
    every node/.style={shape=coordinate},
    edge style/.style={draw=black, thin}
    ]

\coordinate (p0)  at (6,4);
\coordinate (p1)  at (4,6);
\coordinate (p2)  at (7,1);
\coordinate (p3)  at (1,7);
\coordinate (p4)  at (0,0);
\coordinate (p5)  at (4,0);
\coordinate (p6)  at (0,4);
\coordinate (p7)  at (6,0);
\coordinate (p8)  at (0,6);
\coordinate (p9)  at (3,0);
\coordinate (p10) at (0,3);
\coordinate (p11) at (1,0);
\coordinate (p12) at (0,1);
\coordinate (p13) at (5,1);
\coordinate (p14) at (1,5);
\coordinate (p15) at (6,2);
\coordinate (p16) at (2,6);
\coordinate (p17) at (7,2);
\coordinate (p18) at (2,7);
\coordinate (p19) at (2,1);
\coordinate (p20) at (1,2);
\coordinate (p21) at (3,2);
\coordinate (p22) at (2,3);
\coordinate (p23) at (5,0);
\coordinate (p24) at (0,5);
\coordinate (p25) at (7,4);
\coordinate (p26) at (4,7);
\coordinate (p27) at (7,3);
\coordinate (p28) at (6,5);
\coordinate (p29) at (5,6);
\coordinate (p30) at (3,7);
\coordinate (p31) at (3,1);
\coordinate (p32) at (1,3);
\coordinate (p33) at (5,2);
\coordinate (p34) at (2,5);
\coordinate (p35) at (2,0);
\coordinate (p36) at (0,2);
\coordinate (p37) at (4,4);
\coordinate (p38) at (1,1);
\coordinate (p39) at (4,1);
\coordinate (p40) at (1,4);
\coordinate (p41) at (6,1);
\coordinate (p42) at (1,6);
\coordinate (p43) at (2,2);
\coordinate (p44) at (4,3);
\coordinate (p45) at (3,4);
\coordinate (p46) at (4,2);
\coordinate (p47) at (2,4);
\coordinate (p48) at (5,4);
\coordinate (p49) at (4,5);
\coordinate (p50) at (5,5);
\coordinate (p51) at (3,3);
\coordinate (p52) at (6,3);
\coordinate (p53) at (3,6);
\coordinate (p54) at (5,3);
\coordinate (p55) at (3,5);

\draw
 (p6)--(p24) (p6)--(p10) (p10)--(p36) (p12)--(p36) (p4)--(p12) (p8)--(p24)
 (p2)--(p17) (p25)--(p27) (p17)--(p27)
 (p5)--(p23) (p5)--(p9) (p9)--(p35) (p11)--(p35) (p4)--(p11) (p7)--(p23)
 (p3)--(p18) (p26)--(p30) (p18)--(p30)
 (p28)--(p29) (p25)--(p28) (p26)--(p29)
 (p2)--(p7) (p3)--(p8)
 (p48)--(p49) (p37)--(p48) (p37)--(p49)
 (p44)--(p45) (p37)--(p44) (p37)--(p45)
 (p44)--(p51) (p45)--(p51)
 (p48)--(p50) (p49)--(p50)
 (p28)--(p50) (p29)--(p50)

 (p17)--(p41) (p15)--(p41) (p15)--(p17)
 (p18)--(p42) (p16)--(p42) (p16)--(p18)

 (p15)--(p54) (p15)--(p33) (p33)--(p54)
 (p16)--(p55) (p16)--(p34) (p34)--(p55)

 (p2)--(p41) (p3)--(p42)

 (p5)--(p13) (p13)--(p23)
 (p6)--(p14) (p14)--(p24)

 (p19)--(p31) (p21)--(p31) (p19)--(p21)
 (p20)--(p32) (p22)--(p32) (p20)--(p22)

 (p5)--(p39) (p9)--(p39)
 (p6)--(p40) (p10)--(p40)

 (p15)--(p52) (p52)--(p54)
 (p16)--(p53) (p53)--(p55)

 (p13)--(p39) (p13)--(p33) (p33)--(p39)
 (p14)--(p40) (p14)--(p34) (p34)--(p40)

 (p37)--(p54) (p48)--(p54)
 (p37)--(p55) (p49)--(p55)

 (p9)--(p31) (p31)--(p39)
 (p10)--(p32) (p32)--(p40)

 (p31)--(p35) (p32)--(p36)
 (p19)--(p35) (p20)--(p36)

 (p31)--(p46) (p21)--(p46)
 (p32)--(p47) (p22)--(p47)

 (p13)--(p15)
 (p14)--(p16)

 (p0)--(p50) (p0)--(p48)
 (p1)--(p50) (p1)--(p49)

 (p11)--(p19) (p12)--(p20)

 (p11)--(p38) (p19)--(p38)
 (p12)--(p38) (p20)--(p38) (p4)--(p38)

 (p13)--(p41) (p23)--(p41)
 (p14)--(p42) (p24)--(p42)

 (p0)--(p28) (p1)--(p29)

 (p0)--(p52) (p48)--(p52)
 (p1)--(p53) (p49)--(p53)

 (p0)--(p25) (p0)--(p27)
 (p1)--(p26) (p1)--(p30)

 (p19)--(p43) (p21)--(p43)
 (p20)--(p43) (p22)--(p43)

 (p27)--(p52) (p30)--(p53)

 (p38)--(p43)

 (p33)--(p46) (p39)--(p46)
 (p34)--(p47) (p40)--(p47)

 (p33)--(p44) (p44)--(p46)
 (p34)--(p45) (p45)--(p47)

 (p7)--(p41)
 (p8)--(p42)

 (p44)--(p54) (p45)--(p55)

 (p43)--(p51) (p21)--(p51) (p22)--(p51)
 (p46)--(p51) (p47)--(p51)

 (p17)--(p52)
 (p18)--(p53);

\foreach \i in {0,...,55} {
     \fill (p\i) circle (2pt);
}

\end{tikzpicture}
  \hspace{0.5in}
    \begin{tikzpicture}[
    scale=0.6,
    every node/.style={shape=coordinate},
    edge style/.style={draw=black, thin}]

\foreach \pts [count=\i from 0] in {
    (7,1), (1,7), (2,0), (0,2), (6,1), (1,6), (6,5), (5,6), (6,0), (0,6), (7,6),
    (6,7), (4,1), (1,4), (7,2), (2,7), (6,3), (3,6), (7,3), (3,7), (3,0), (0,3),
    (4,2), (2,4), (5,5), (3,2), (2,3), (7,4), (4,7), (5,3), (3,5), (5,2), (2,5),
    (3,1), (1,3), (5,0), (0,5), (6,4), (4,6), (4,0), (0,4), (5,4), (4,5), (7,5),
    (5,7), (3,3), (2,2), (5,1), (1,5), (4,4), (6,2), (2,6), (2,1), (1,2), (6,6),
    (4,3), (3,4), (1,1)
} {
    \node (p\i) at \pts {};
         \fill (p\i) circle (2pt);
}

\def\connect(#1,#2){\draw[edge style] (p#1) -- (p#2);}

\foreach \edge in {
    (3,21), (9,36), (21,40), (36,40), (0,14), (18,27), (27,43), (14,18), (10,43),
    (2,20), (8,35), (20,39), (35,39), (1,15), (19,28), (28,44), (15,19), (11,44),
    (0,8), (1,9), (10,11), (2,3), (52,53), (2,52), (3,53), (25,26), (25,46),
    (26,46), (0,4), (4,8), (1,5), (5,9), (4,14), (5,15), (25,45), (26,45), (46,52),
    (46,53), (10,54), (11,54), (6,41), (6,24), (24,41), (7,42), (7,24), (24,42),
    (33,46), (25,33), (34,46), (26,34), (6,37), (37,41), (7,38), (38,42), (16,29),
    (16,31), (29,31), (17,30), (17,32), (30,32), (12,25), (22,25), (12,22),
    (13,26), (23,26), (13,23), (20,52), (21,53), (16,37), (29,37), (17,38),
    (30,38), (16,50), (31,50), (17,51), (32,51), (33,52), (34,53), (12,33),
    (13,34), (24,54), (6,54), (7,54), (16,27), (27,37), (17,28), (28,38), (22,31),
    (22,29), (23,32), (23,30), (24,49), (41,49), (42,49), (16,18), (17,19),
    (22,45), (23,45), (37,43), (6,43), (38,44), (7,44), (29,41), (30,42), (22,47),
    (12,47), (23,48), (13,48), (4,35), (5,36), (33,39), (12,39), (34,40), (13,40),
    (45,49), (45,55), (49,55), (45,56), (49,56), (31,47), (32,48), (4,47), (35,47),
    (5,48), (36,48), (4,50), (14,50), (5,51), (15,51), (29,55), (22,55), (30,56),
    (23,56), (20,33), (21,34), (47,50), (48,51), (39,47), (40,48), (18,50),
    (19,51), (41,55), (42,56), (43,54), (44,54)
} {
    \expandafter\connect\edge
}

\end{tikzpicture}
  \caption{Projection of the concave hulls for the knots \texttt{11a100} and
  \texttt{11a165}.  The figure on the right shows a subdivision which is 
  not a triangulation, as it contains a non-simplicial tile on the lower left.}
  \label{fig.11a100}
\end{figure}
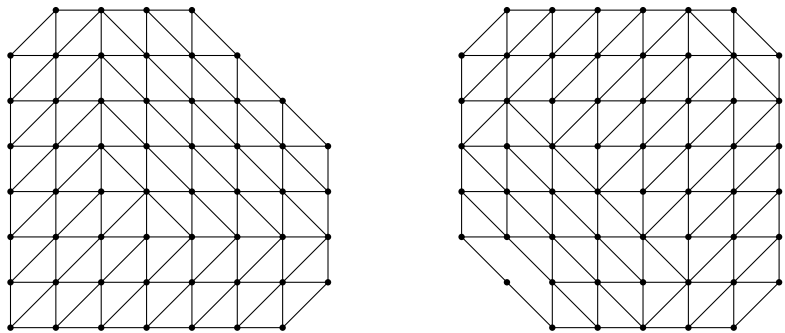

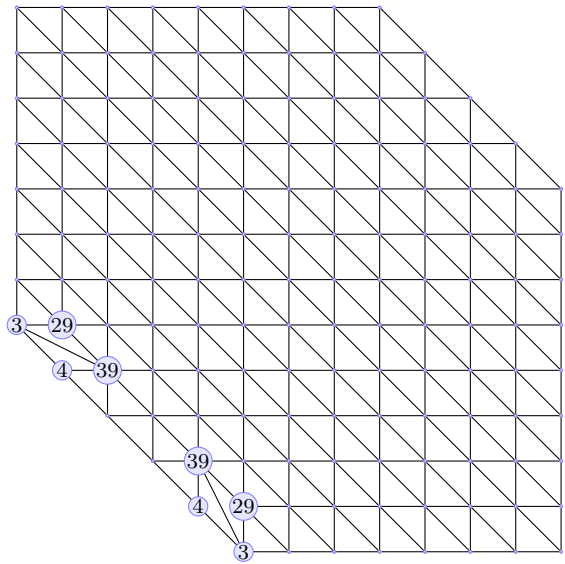
\begin{figure}[htbp]
     \centering
\begin{tikzpicture}[
    scale=0.6, 
    every node/.style={circle, fill=blue!10, draw=blue!50, inner sep=0.2pt, font=\tiny}
]
\foreach \p [count=\i from 0] in {(12,0), (0,12), (5,0), (0,5), (11,1), (1,11), (4,2), (2,4), (6,1), (1,6), (5,1), (1,5), (11,0), (12,1), (3,2), (2,3), (0,11), (1,12), (12,6), (6,12), (6,5), (5,6), (10,10), (12,4), (4,12), (7,1), (1,7), (6,0), (12,2), (0,6), (2,12), (11,3), (3,11), (4,4), (4,1), (1,4), (12,8), (8,12), (10,0), (12,7), (0,10), (7,12), (9,3), (3,9), (10,1), (1,10), (12,3), (3,12), (7,6), (6,7), (11,2), (2,11), (5,5), (10,3), (3,10), (6,6), (10,8), (8,10), (8,1), (1,8), (7,3), (3,7), (11,5), (5,11), (9,8), (8,9), (9,2), (2,9), (8,0), (0,8), (8,6), (6,8), (4,3), (3,4), (9,9), (7,7), (8,8), (8,2), (2,8), (11,6), (6,11), (10,9), (9,10), (8,4), (4,8), (9,1), (1,9), (10,7), (7,10), (11,8), (8,11), (9,5), (5,9), (8,5), (5,8), (11,4), (4,11), (7,2), (2,7), (5,2), (2,5), (5,3), (3,5), (6,3), (3,6), (7,5), (5,7), (3,3), (8,7), (7,8), (7,4), (4,7), (8,3), (3,8), (5,4), (4,5), (9,7), (7,9), (9,6), (6,9), (11,9), (9,11), (11,7), (7,11), (9,4), (4,9), (6,4), (4,6), (10,2), (2,10), (10,4), (4,10), (7,0), (9,0), (0,7), (0,9), (10,5), (5,10), (6,2), (2,6), (12,5), (5,12), (10,6), (6,10)} {
        \node (p\i) at \p {};
}
\foreach \u/\v in {1/16, 16/40, 3/29, 40/135, 29/134, 69/134, 69/135, 0/13, 18/39, 28/46, 13/28, 23/46, 36/39, 23/140, 18/140, 0/12, 12/38, 2/27, 38/133, 27/132, 68/132, 68/133, 14/15, 2/34, 3/35, 14/34, 15/35, 1/17, 19/41, 30/47, 17/30, 24/47, 37/41, 24/141, 19/141, 22/120, 22/121, 36/120, 37/121, 64/65, 64/74, 65/74, 81/82, 22/81, 22/82, 114/115, 52/114, 52/115, 48/49, 48/75, 49/75, 20/21, 20/55, 21/55, 20/52, 21/52, 48/55, 49/55, 14/107, 15/107, 72/73, 33/72, 33/73, 0/4, 4/12, 4/13, 1/5, 5/16, 5/17, 64/76, 65/76, 33/114, 33/115, 108/109, 75/108, 75/109, 76/108, 76/109, 74/81, 74/82, 72/107, 73/107, 38/44, 12/44, 40/45, 16/45, 44/85, 38/85, 45/86, 40/86, 56/64, 56/74, 57/65, 57/74, 66/85, 44/66, 67/86, 45/67, 48/70, 70/75, 49/71, 71/75, 52/126, 114/126, 52/127, 115/127, 50/53, 31/53, 31/50, 51/54, 32/54, 32/51, 60/77, 60/97, 77/97, 61/78, 61/98, 78/98, 55/105, 20/105, 55/106, 21/106, 60/112, 77/112, 61/113, 78/113, 6/107, 6/14, 7/107, 7/15, 39/122, 39/89, 89/122, 41/123, 41/90, 90/123, 6/99, 72/99, 6/72, 7/100, 73/100, 7/73, 60/110, 110/112, 61/111, 111/113, 18/122, 19/123, 110/126, 60/126, 111/127, 61/127, 2/10, 6/10, 2/6, 3/11, 7/11, 3/7, 10/27, 11/29, 20/126, 21/127, 48/105, 49/106, 48/93, 70/93, 49/94, 71/94, 64/87, 64/116, 87/116, 65/88, 65/117, 88/117, 8/10, 8/27, 9/11, 9/29, 83/112, 42/112, 42/83, 84/113, 43/113, 43/84, 20/110, 105/110, 21/111, 106/111, 10/99, 8/99, 11/100, 9/100, 56/87, 57/88, 76/116, 76/117, 81/120, 82/121, 93/105, 94/106, 50/128, 53/128, 51/129, 54/129, 53/130, 31/130, 54/131, 32/131, 4/44, 5/45, 70/108, 71/109, 70/91, 91/93, 71/92, 92/94, 4/50, 13/50, 5/51, 17/51, 8/25, 8/138, 25/138, 9/26, 9/139, 26/139, 83/110, 84/111, 72/101, 33/101, 73/102, 33/102, 42/66, 66/112, 43/67, 67/113, 25/58, 58/97, 25/97, 26/59, 59/98, 26/98, 83/105, 83/93, 84/106, 84/94, 42/124, 83/124, 43/125, 84/125, 42/53, 53/124, 43/54, 54/125, 99/101, 100/102, 101/114, 102/115, 4/128, 44/128, 5/129, 45/129, 28/31, 28/50, 30/32, 30/51, 70/118, 91/118, 71/119, 92/119, 85/133, 86/135, 25/68, 58/68, 26/69, 59/69, 60/103, 97/103, 61/104, 98/104, 42/128, 43/129, 66/77, 77/85, 67/78, 78/86, 8/132, 25/132, 9/134, 26/134, 66/128, 67/129, 108/118, 109/119, 93/124, 94/125, 6/34, 7/35, 58/77, 59/78, 31/46, 32/47, 103/126, 104/127, 58/85, 59/86, 108/116, 109/117, 124/130, 125/131, 62/95, 95/136, 62/136, 63/96, 96/137, 63/137, 23/95, 23/62, 24/96, 24/63, 91/124, 92/125, 56/81, 57/82, 31/95, 46/95, 32/96, 47/96, 99/138, 100/139, 95/130, 96/131, 103/114, 101/103, 104/115, 102/104, 97/138, 98/139, 79/122, 18/79, 80/123, 19/80, 116/118, 117/119, 87/122, 56/122, 88/123, 57/123, 56/89, 81/89, 57/90, 82/90, 36/89, 37/90, 130/136, 131/137, 91/130, 92/131, 79/87, 80/88, 91/136, 118/136, 92/137, 119/137, 89/120, 90/121, 103/138, 104/139, 87/142, 116/142, 88/143, 117/143, 101/138, 102/139, 62/140, 63/141, 79/140, 80/141, 58/133, 59/135, 62/79, 79/142, 62/142, 63/80, 80/143, 63/143, 118/142, 119/143, 136/142, 137/143} {
    \draw (p\u) -- (p\v);
}
\foreach \l [count=\i from 0] in {, , 3, 3, , , 39, 39, , , 29, 29, , , , , , , , , , , , , , , , , , , , , , , 4, 4, , , , , , , , , , , , , , , , , , , , , , , , , , , , , , , , , , , , , , , , , , , , , , , , , , , , , , , , , , , , , , , , , , , , , , , , , , , , , , , , , , , , , , , , , , , , , , , , , , , , , , , , , , , , } {
    \node at (p\i) {\l};
}
\end{tikzpicture}
         \caption{Projection of the concave hull for the knot {\tt 15a57208}.  
         The subdivision contains edges with slopes $-2$ and $-\frac{1}{2}$.   
         Labels show some coefficients of $\LG(-t_1,-t_2)$.}
         \label{fig.15a57208}
\end{figure}

However, a small number of non-type-B knots exist, whose $\LG$-coefficients are 
log-concave, but not type-B log-concave.
A sample non-type-B subdivision is shown in \Cref{fig.15a57208}. Inspection of 
this and subdivisions of the remaining 544 non-type-B knots up to 19 crossings 
shows that their edges have slopes $0,\infty,1,-1$ (as type-B have) or 
$2,-2,1/2,-1/2$. 

This motivates the question~\eqref{conj.slope} stated in the introduction. We 
caution, however, that due to the small number of non-type-B knots, the found 
slopes may not be representative for alternating knots with more than 19 crossings.

\begin{question}
\label{conj.slope}
For every alternating knot $K$, do the slopes of edges of the subdivision induced 
by the log coefficients of $\LG_K(-t_1,-t_2)$ belong to 
$\{0, \infty, \pm 1, \pm 2, \pm \frac{1}{2}\}$? 
\end{question}

\subsection{Non-simplicial faces of the subdivision}
\label{subsec.nonsimplicial}

If the coefficients are sufficiently generic numbers, the tiles in the subdivision 
are expected to be standard triangles of area $1/2$, as in the case of the 
knot \texttt{11a100} whose subdivision
is shown in Figure~\ref{fig.11a100}. 
However in many cases the subdivision contains tiles other than the
standard triangles, which we call {\em non-simplicial faces}. Although we have 
not listed all such knots up to 19 crossings,
\num{163931} out of \num{2261237} type-B alternating knots up to 17 crossings have
subdivisions that contains tiles other than standard triangles, and the ratio for 
non-type-B alternating knots up to 17 crossings is \num{24} out of \num{69}. See
Figure~\ref{fig.11a100} (right) and Figures~\ref{fig.11a364} for examples.

\begin{figure}[htbp!]
     \centering
     \begin{tabular}{lr}
        \begin{tikzpicture}[
    scale=0.6,
    every node/.style={shape=coordinate},
    edge style/.style={draw=black, thin}]

\foreach \pts [count=\i from 0] in {
 (9,1), (9,2), (1,9), (2,9), (2,2), (5,2), (2,5), (0,0), (8,0), (0,8), (2,1), (6,1),
       (1,2), (1,6), (8,1), (1,8), (1,0), (7,0), (0,1), (0,7), (7,2), (2,7), (6,2), (8,2),
       (2,6), (2,8), (1,1), (7,1), (1,7), (4,4), (4,5), (4,6), (3,0), (4,7), (3,1), (3,2),
       (3,3), (3,4), (3,5), (3,6), (3,7), (2,0), (3,8), (2,3), (2,4), (1,3), (1,4), (1,5),
       (0,2), (0,3), (0,4), (0,5), (0,6), (8,3), (7,3), (7,4), (6,0), (6,3), (6,4), (6,5),
       (5,0), (5,1), (5,3), (5,4), (5,5), (5,6), (4,0), (4,1), (4,2), (4,3)
} {
 \fill \pts circle (2pt);
    \node (p\i) at \pts {};
}

\def\connect(#1,#2){\draw[edge style] (p#1) -- (p#2);}

\foreach \edge in {
(18,19), (7,18), (9,19), (16,17), (7,16), (8,17), (2,9), (0,8), (0,1), (2,3),
       (1,3), (4,6), (4,5), (5,6), (1,14), (0,14), (3,15), (2,15), (1,23), (3,25), (23,25),
       (12,13), (6,13), (4,12), (10,11), (5,11), (4,10), (22,24), (5,22), (6,24), (12,18),
       (13,19), (10,16), (11,17), (14,23), (15,25), (8,14), (9,15), (20,21), (20,22),
       (21,24), (20,23), (21,25), (7,26), (16,26), (17,27), (8,27), (18,26), (19,28),
       (9,28), (10,26), (11,27), (12,26), (13,28), (11,22), (13,24), (4,26), (14,27),
       (15,28), (14,20), (15,21), (22,27), (24,28), (20,27), (21,28)
} {
    \expandafter\connect\edge
}

\end{tikzpicture} &
\begin{tikzpicture}[
    scale=0.8,
    every node/.style={circle, fill=blue!20, draw=blue!80, inner sep=1pt, font=\tiny},
    edge style/.style={draw=black!60, thin}
    ]

\foreach \pts [count=\i from 0] in {
    (9,1), (9,2), (1,9), (2,9), (2,2), (5,2), (2,5), (0,0), (8,0), (0,8), (2,1), (6,1),
       (1,2), (1,6), (8,1), (1,8), (1,0), (7,0), (0,1), (0,7), (7,2), (2,7), (6,2), (8,2),
       (2,6), (2,8), (1,1), (7,1), (1,7), (4,4), (4,5), (4,6), (3,0), (4,7), (3,1), (3,2),
       (3,3), (3,4), (3,5), (3,6), (3,7), (2,0), (3,8), (2,3), (2,4), (1,3), (1,4), (1,5),
       (0,2), (0,3), (0,4), (0,5), (0,6), (8,3), (7,3), (7,4), (6,0), (6,3), (6,4), (6,5),
       (5,0), (5,1), (5,3), (5,4), (5,5), (5,6), (4,0), (4,1), (4,2), (4,3)
} {
    \node (p\i) at \pts {};
}

\def\connect(#1,#2){\draw[edge style] (p#1) -- (p#2);}

\foreach \edge in {
(18,19), (7,18), (9,19), (16,17), (7,16), (8,17), (2,9), (0,8), (0,1), (2,3),
       (1,3), (4,6), (4,5), (5,6), (1,14), (0,14), (3,15), (2,15), (1,23), (3,25), (23,25),
       (12,13), (6,13), (4,12), (10,11), (5,11), (4,10), (22,24), (5,22), (6,24), (12,18),
       (13,19), (10,16), (11,17), (14,23), (15,25), (8,14), (9,15), (20,21), (20,22),
       (21,24), (20,23), (21,25), (7,26), (16,26), (17,27), (8,27), (18,26), (19,28),
       (9,28), (10,26), (11,27), (12,26), (13,28), (11,22), (13,24), (4,26), (14,27),
       (15,28), (14,20), (15,21), (22,27), (24,28), (20,27), (21,28)
} {
    \expandafter\connect\edge
}

\foreach \l [count=\i from 0] in {
1, 1, 1, 1, 18, 18, 18, 3, 3, 3, 14, 14, 14, 14, 6, 6, 5, 5, 5, 5, 10, 10, 16, 4,
       16, 4, 12, 12, 12, 16, 10, 4, 5, 1, 14, 18, 18, 18, 16, 10, 4, 5, 1, 18, 18, 14, 14,
       14, 5, 5, 5, 5, 5, 1, 4, 1, 5, 10, 4, 1, 5, 14, 16, 10, 4, 1, 5, 14, 18, 18
} {
    \node at (p\i) {\l};
}

\end{tikzpicture}
     \end{tabular}
         \caption{Projection of the concave hull for the type-B knot {\tt 11a364}. 
         Labels show the coefficients of $\LG(-t_1,-t_2)$.}
         \label{fig.11a364}
     \end{figure}
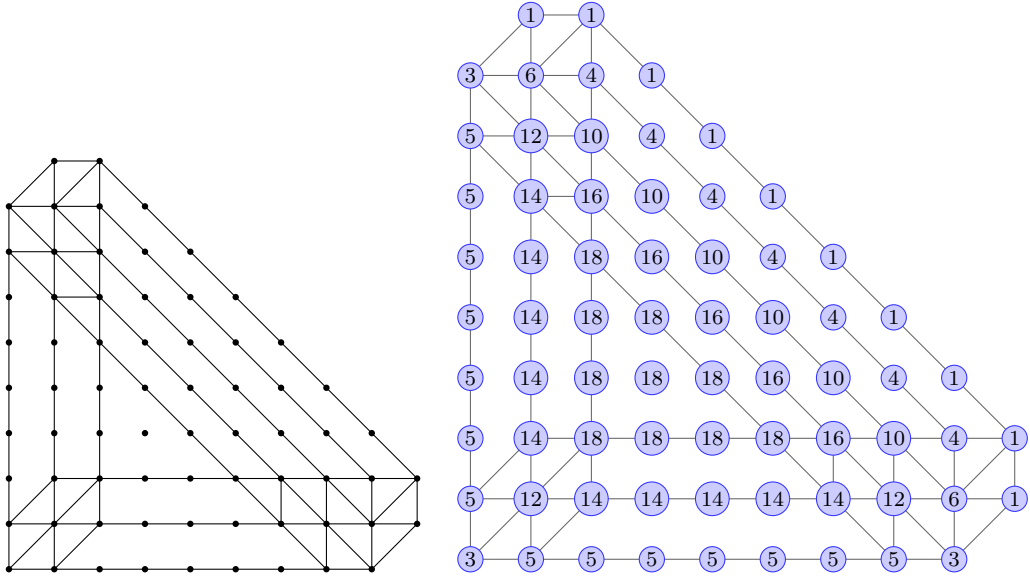

Moreover, from the limited data we observe that if a two dimensional face has a point 
in its interior, then it also has a point in the relative interior of one of its edge. 
For example, the following configurations cannot be faces in the subdivision although
they have the expected slopes.
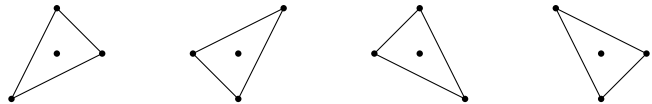
\begin{equation}
 \label{motzkin}   
             \begin{tikzpicture}[
    scale=0.6,
    every node/.style={shape=coordinate},
    edge style/.style={draw=black, thin}]

\foreach \pts [count=\i from 0] in {
   (0,0),
   (1,1),
   (2,1),
   (1,2),
   (4,1),
   (5,1),
   (5,0),
   (6,2),
   (10,0),
   (9,1),
   (8,1),
   (9,2),
   (12,2),
   (13,1),
   (13,0),
   (14,1)
} {
    \node (p\i) at \pts {};
         \fill (p\i) circle (2pt);
}

\def\connect(#1,#2){\draw[edge style] (p#1) -- (p#2);}

\foreach \edge in {
   (0,2),
   (0,3),
   (2,3),
   (4,6),
   (4,7),
   (6,7),
   (8,10),
   (8,11),
   (10,11),
   (12,14),
   (12,15),
   (14,15)
} {
    \expandafter\connect\edge
}

\end{tikzpicture} 
 \end{equation}

\subsection{Alexander versus Links--Gould polynomial concavity}

In this section we comment on the relationship between $\LG$-polynomial concavity
and the Alexander-polynomial concavity. This is summarized as the following relation:
\be
\label{eq.LGA}
\LG_K(-t_1,-t_2)\text{-log-concave} \not\Rightarrow \Delta_K(-t)^2\text{-log-concave}
\Leftarrow \Delta_K(-t)\text{-log-concave}
\ee
which we now explain. 
Let $\Delta_K(t)$ be the Alexander polynomial of an alternating knot $K$; it is known
that the coefficients of $\Delta_K(-t)$ are always nonnegative \cite{Crowell,Murasugi}. 
Fox conjectured
~\cite{Fox} that the coefficients $a_i$ of $\Delta_K(-t)$ satisfy the {\em trapezoidal}
property 
\begin{equation}
\label{eqn:trapezoid}
    0 = a_{-n-1}<a_{-n} < a_{-n+1}<\cdots  a_k = \cdots = a_m > a_{m+1}>\cdots 
    >a_n>a_{n+1}=0
\end{equation}
for some $-n\leq k\leq m\leq n$, and Stoimenow further conjectured~\cite{Stoimenow} 
that 
they are log-concave, that is for any three consecutive coefficients 
$a_{i-1}, a_i, a_{i+1}$, we have 
\begin{equation}
\label{eqn:logconcave}
a_{i-1}\cdot a_{i+1} \leq a_i^2 \,.
\end{equation}
It is well-known that log-concavity implies the condition~\eqref{eqn:trapezoid}.

The Fox--Stoimenow conjecture, together with the fact that products of univariate
polynomials with log-concave positive coefficients also have log-concave
coefficients~\cite{Hoggar}, implies that 
the coefficients of $\Delta_K(-t)^2$ are log-concave. On the other hand,
Equation~\eqref{eq.2special} 
implies that $\Delta_K(-t)^2=\LG_K(-t,-t^{-1})$, that is, the coefficients of 
$\Delta_K(-t)^2$ are sums of coefficients of terms in $\LG_K(-t,-t^{-1})$ along the
diagonals in direction $(1,1)$. Even though diagonal specializations of continuous
functions preserve log-concavity~\cite[Theorem 6]{Prkopa}, a diagonal specialization 
of discrete log-concave functions need not be log-concave, even for the type-B ones.
For example, the following configuration on the plane is type-B log-concave, but its
diagonal sum gives the sequence $1, 4, 17, 4, 1$, which is not log-concave; rotating
the picture $90^\circ$, we obtain a diagonal sum sequence $4,4,11,4,4$ which gives
another counterexample.
 
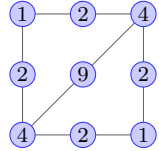
\begin{center}
             \begin{tikzpicture}[
    scale=0.8,
    every node/.style={circle, fill=blue!20, draw=blue!80, inner sep=1pt, font=\tiny},
    edge style/.style={draw=black!60, thin}
    ]

\foreach \pts [count=\i from 0] in {
   (0,0),
   (1,1),
   (2,2),
   (0,1),
   (1,0),
   (2,1),
   (1,2),
   (0,2),
   (2,0)
} {
    \node (p\i) at \pts {};
}

\def\connect(#1,#2){\draw[edge style] (p#1) -- (p#2);}

\foreach \edge in {
(0,7),
(0,8),
(2,7),
(2,8),
(0,2)
} {
    \expandafter\connect\edge
}

\foreach \l [count=\i from 0] in {
4,9,4,2,2,2,2,1,1
} {
    \node at (p\i) {\l};
}

\end{tikzpicture} 
 \end{center} 
This explains the non-implication on the left of Equation~\eqref{eq.LGA}.
This suggests that, there is a stronger stronger version of Conjecture~\ref{conj.1}
which is compatible with the Fox--Stoimenow conjecture.

We end this section with an observation relating the nonsimplicial faces discussed 
in the previous section to Alexander polynomials. We will say that a knot is
\textit{properly trapezoidal} if $k < m$ in condition~\eqref{eqn:trapezoid}. In the
subdivision induced by log LG-coefficients, the existence of nonsimiplicial faces 
means that there are lifted points near each other that are affinely dependent, and 
an analogous property for the Alexander coefficients would be for the log-concavity
property~\eqref{eqn:logconcave} to be attained at equality.

Among the alternating knots with up to 17 crossings, there are \num{4118} knots
satisfying the inequality~\eqref{eqn:logconcave} at equality for some $i$, which 
are precisely the ones whose Alexander polynomials are properly trapezoidal.  
Moreover, all these \num{4118} knots have nonsimplicial faces in the sense that we 
discussed in Section~\ref{subsec.nonsimplicial}. This gives us a single-way 
implication, recalling that there are \num{163931} alternating knots with 
nonsimpilicial faces up to 17 crossings.

\begin{question}
\label{conj.trapezoid}
If an alternating knot $K$ has properly trapezoidal Alexander polynomial, does it
always have a nonsimplicial face in the subdivision induced by the log coefficients 
of $\LG_K(-t_1,-t_2)$? 
\end{question}


\section{Type-B concavity}
\label{sec.typeB}

In this section we introduce discrete concave functions of type-B and show that they
are concave extendable.  Although it is a stronger condition, type-B concavity is
easier to check than concave extendability.  As stated in Theorem~\ref{thm.1}, we 
show computationally that for 99.999\%, of alternating knots with up to 19 crossings,
the LG coefficients satisfy type-B log concavity, which implies log-concavity.  In 
this section and the next we explain the theory and algorithms behind the computation
that proves Theorem~\ref{thm.1}.

\subsection{Discrete convex sets of type-B}
\label{sub.basicsB}

 We will now define discrete convex sets of type-B, also known
as type-B $M^\natural$-convex sets, following Murota's work in type A; see Chapters 4
and 6 of~\cite{Murota}.  This will prepare us to discuss type-B concave functions in
the next subsection.
Let $e_1,\dots,e_n$ denote the standard basis vectors in $\Z^n$.  

\begin{definition}
\label{def.typeBset}
  A nonempty subset $A\subset \Z^n$ is called {\em $M^\natural$-convex of type-B}
  (we will call B-convex, for short) if it satisfies the following exchange axiom:
\begin{quote}
  For $x,y\in A$ and coordinate $i\in [n]$ such that $x_i > y_i$, one of the following
  holds:
\begin{enumerate}
    \item[(1)] the points $x - e_i$ and $y + e_i$ are both in $A$;
    \item[(2)] there exists a coordinate $j$ with $x_j < y_j$ such that $x-e_i+e_j$
      and $y+e_i-e_j$ are both in $A$;
    \item[(3)] there exists a coordinate $j \neq i$ with $x_j > y_j$ such that
      $x-e_i-e_j$ and $y+e_i+e_j$ are both in $A$.
\end{enumerate}
\end{quote}
\end{definition}

In other words, whenever there are two distinct points in the set $A$, we can walk
them closer to each other along the vectors in a type-B root system
($e_i$, $e_i\pm e_j$) while remaining in the set $A$.
The first two conditions give the classical (type $A$) $M^\natural$ convexity.  

Let $\conv(A)$ denotes the convex hull of $A$. The following is immedate from the
definition.

\begin{lemma}
\label{lem:dir}
If $A\subset \Z^n$ is B-convex and $\conv(A)$ is $1$-dimensional, then
$\conv(A)$ is parallel to $e_i$ or $e_i\pm e_j$.
\end{lemma}

\begin{proposition}
\label{prop.nohole}
The B-convex sets are hole-free, i.e.\ for any B-convex set $A$, we have
$\conv(A) \cap \Z^n = A.$  
\end{proposition}

The proof follows a similar argument as in Theorem 4.12 of Murota's book~\cite{Murota}. 

\begin{proof}
It is clear that $A \subset \conv(A) \cap \Z^n$.  For the reverse
inclusion let $x \in \conv(A)\cap\Z^n$.  We can write

\begin{equation}
\label{eqn:convComb}
x = \sum_{k=1}^m \lambda_k x^{(k)} \text{ where } x^{(k)} \in A,
\sum_{k=1}^m \lambda_k = 1, \text{ and } \lambda_k > 0 \text{ for }k=1,\dots,m.
\end{equation}  
Let $\Phi_i = \sum_{k=1}^m \lambda_k |x^{(k)}_i - x_i|$, be a measure of  complexity
of the representation~\eqref{eqn:convComb} for $i$th coordinate.  If $\Phi_i = 0$,
then all points $x^{(k)}$ have the same $i$th coordinate as $x$.
If $\Phi_i > 0$, then there are two points $x^{(k)},x^{(\ell)}$ such that
$x^{(k)}_i > x_i > x^{(\ell)}_i$.  Applying the exchange axiom  in 
Definition~\ref{def.typeBset} to $x^{(k)}$
and $x^{(\ell)}$ with coordinate $i$ gives one of the following.

\begin{enumerate}
\item[(1)]
  The points $x^{(k)}-e_i$ and $x^{(\ell)}+e_i$ are both in $A$.
  In \eqref{eqn:convComb} replace $\lambda_k x^{(k)} + \lambda_\ell x^{(\ell)}$ with:
  \begin{itemize}
  \item
    $\lambda_k(x^{(k)}-e_i) + \lambda_k (x^{(\ell)}+e_i)
     + (\lambda_\ell-\lambda_k) x^{(\ell)}$ if $\lambda_k < \lambda_\ell$,
   \item
     $\lambda_\ell(x^{(k)}-e_i) + \lambda_\ell (x^{(\ell)}+e_i)
     + (\lambda_k-\lambda_\ell) x^{(k)}$ if $\lambda_k > \lambda_\ell$.
  \end{itemize}
\item[(2)]
  There is a coordinate $j$ with $x^{(k)}_j < x^{(\ell)}_j$ such that
  $x^{(k)}-e_i+e_j$ and $x^{(\ell)}+e_i-e_j$ are both in $A$.  In \eqref{eqn:convComb}
  replace $\lambda_k x^{(k)} + \lambda_\ell x^{(\ell)}$ with:
  \begin{itemize}
  \item
    $\lambda_k(x^{(k)}-e_i+e_j) + \lambda_k (x^{(\ell)}+e_i-e_j)
    + (\lambda_\ell-\lambda_k) x^{(\ell)}$ if $\lambda_k < \lambda_\ell$,
  \item
    $\lambda_\ell(x^{(k)}-e_i+e_j) + \lambda_\ell (x^{(\ell)}+e_i-e_j)
    + (\lambda_k-\lambda_\ell) x^{(k)}$ if $\lambda_k > \lambda_\ell$.
  \end{itemize}
\item[(3)]
  There is a coordinate $j\neq i$ with $x^{(k)}_j > x^{(\ell)}_j$ such that
  $x^{(k)}-e_i-e_j$ and $x^{(\ell)}+e_i+e_j$ are both in $A$.  In
  \eqref{eqn:convComb} replace $\lambda_k x^{(k)} + \lambda_\ell x^{(\ell)}$ with:
  \begin{itemize}
  \item
    $\lambda_k(x^{(k)}-e_i-e_j) + \lambda_k (x^{(\ell)}+e_i+e_j)
    + (\lambda_\ell-\lambda_k) x^{(\ell)}$ if $\lambda_k < \lambda_\ell$,
  \item
    $\lambda_\ell(x^{(k)}-e_i-e_j) + \lambda_\ell (x^{(\ell)}+e_i+e_j)
    + (\lambda_k-\lambda_\ell) x^{(k)}$ if $\lambda_k > \lambda_\ell$.
  \end{itemize}
\end{enumerate}

Let $N$ be a positive integer such that $N\lambda_k$ is an integer for all
$k=1,\dots,m$. Each of the exchanges above reduces the complexity $\Phi_i$ by at
least $2 \min(\lambda_k,\lambda_\ell)\geq 2/N$, while preserving the condition
that $N\lambda_k\in \Z$ for all $k=1,\dots,m$.  Thus after finitely many exchanges,
we obtain $\Phi_i =0$.  Moreover, if $\Phi_j$ is already zero for some $j$, then
the exchanges do not affect $\Phi_j$.  Repeating this for each coordinate, we have
$\Phi_i = 0$ for all $i=1,\dots,n$, which implies that $x \in A$.
\end{proof}

\subsection{Discrete  concave functions of type-B}
\label{sub.typeBfun}

We define the {\em effective domain} of a function
$f : \Z^n \rightarrow \R \cup \{-\infty\}$
to be 
\[
\dom(f) := \{x \in \Z^n \mid f(x) \neq -\infty\}.
\]  
For simplicity, we will assume here at $\dom(f)$ is a finite set, although the
statements should extend to the infinite case, with some care.

\begin{definition}
\label{def.typeB}
A function $f : \Z^n \rightarrow \R \cup \{-\infty\}$ with $\dom(f)\neq \varnothing$
is called {\em $M^\natural$-concave of type-B} (or {\em B-concave}, for short),
if it satisfies the following exchange condition.
\begin{quote}
  For $x,y\in \dom(f)$ and coordinate $i\in [n]$ such that $x_i > y_i$, one of the
  following holds:
\begin{enumerate}
\item[(1)]
  $f(x)+f(y) \leq f(x - e_i) + f(y + e_i)$;
\item[(2)]
  there exists $j$ with $x_j < y_j$ such that
  $f(x)+f(y) \leq f(x-e_i+e_j)+f(y+e_i-e_j)$;
\item[(3)]
  there exists $j \neq i$ with $x_j > y_j$ such that
  $f(x)+f(y)\leq f(x-e_i-e_j)+f(y+e_i+e_j)$.
\end{enumerate}
\end{quote}
\end{definition}

The following is immediate from the definitions.

\begin{lemma}
The effective domain of a B-concave function is a B-convex set.
\end{lemma}

Let $f : \Z^n \rightarrow \R$ be a function.  For $x\in \dom(f)$, the point
$\widetilde{x} := (x, f(x))\in \R^{n+1}$ is called the {\em lift} of $x$, and
let $\widetilde{\dom}(f) := \{\widetilde{x} \mid x \in \dom(f)\}$ be the set
of lifted points. 
An {\em face} of the polytope $P := \conv(\widetilde{\dom}(f))$ in direction
$w \in \R^{n+1}$ is the set
\[
\face_w(P) := \{p \in P \mid w\cdot p \geq w\cdot q \text{ for all } q \in P\}. 
\]
If $w_{n+1} > 0$, we call $\face_w(P)$ an {\em upper face}.
The {\em upper hull} or {\em concave hull} of $\widetilde{\dom}(f)$ is the
union of all upper faces of $P$. We may now prove the following:

\begin{theorem}
\label{thm:concave extendable}
The B-concave functions are concave extendable.
\end{theorem}
 That is, if $f$ is B-concave,
then there is a concave function $g : \conv(\dom(f)) \rightarrow \R$ such that
$f(x) = g(x)$ for all $x \in \dom(f)$.   Equivalently, the condition~\eqref{hconv} 
is satisfied.
The concave extension $g$ is not unique, but there is a unique pointwise-smallest 
one whose graph is the upper hull of the lifted points.   

The following lemma will be used in the proof of the theorem.

\begin{lemma}
\label{lem:concave-hull}
The function $f$ is concave extendable if and only if $\widetilde{\dom}(f)$ is
contained in its upper hull. 
\end{lemma}

\begin{proof}[Proof of Lemma~\ref{lem:concave-hull}]
The concave hull of $\widetilde{\dom}(f)$ is the graph of the 
point-wise smallest concave function on $\conv\dom(f)$ that is no less than $f$ 
on $\dom(f)$.
\end{proof}

\begin{proof}[Proof of \Cref{thm:concave extendable}]
By the lemma above, we need to check that all the lifted points in 
$\widetilde{\dom}(f)$ lie on the upper faces of $P$.

Let $u \in \R^n$ and $w = (u,1) \in \R^{n+1}$.  Let
$M = \max\{w\cdot \widetilde{x} \mid x \in \dom(f)\}$.  
Define another function $h : \Z^n \rightarrow \R$ by
\[
h(x) := f(x) + u \cdot x - M.
\]
It is easy to check that $h$ is B-concave, as  $f$ is. 
For all $x \in \dom(f)$ we have
$h(x) = w\cdot \widetilde{x} - M \leq 0$, and $h(x) = 0$ if and only if
$\widetilde{x}$ belongs to the upper face $\face_w(P)$. 

Applying the condition for $h$ being B-concave to points with zero $h$-values,
we see that they can only be exchanged to other points with zero $h$-values as well.
It follows that the points $\{x \in \Z^n \mid h(x) = 0\}$, which is the projection
of $\face_w(P) \cap \widetilde{\dom}(f)$, forms a B-convex set, and in particular,
it is hole-free.  So we conclude that every point is lifted to the upper hull, and
this concludes the proof.
\end{proof}

Recall that a function $f: \Z^n \rightarrow \R$ induces a  subdivision on $\dom(f)$, 
given by the projection of the upper faces of the lift $\conv(\widetilde{\dom}(f))$.

\begin{proposition}
\label{pr:B-slope}
For a B-concave function $f$, the edges of the  subdivision of $\dom(f)$ induced by 
$f$ are parallel to $e_i$ or $e_i \pm e_j$.
\end{proposition}

\begin{proof}
The edges are one-dimensional faces.
The statement follows from Lemma~\ref{lem:dir} and the fact that the projection
of lifted points in an upper face form a B-convex set, as shown in the proof
of the theorem above.
\end{proof}

\subsection{Certifying type-B log concavity}
\label{sec:logBcheck}

For a function $g : A \rightarrow \N$ where $A \subset \Z^n$ is a finite set, it 
is straightforward to check using exact arithmetic that $\log g : A \rightarrow 
\R\cup\{-\infty\}$ is B-concave,  by applying   Definition~\ref{def.typeB} directly,  which takes $O(n^2 |A|^2)$ steps. An inequality of the form \[\log g(x)+\log g(y) 
\leq \log g(x - e_i) + \log g(y + e_i)\] is equivalent to \[g(x) g(y) \leq g(x-e_i)
g(y+e_i),\] and similarly for the other two types of inequalities, so we never have 
to compute the logarithm values.  An explicit algorithm is written out in
Algorithm~\ref{alg.typeB}.

\begin{algorithm}[h]
\caption{Certify type-B log-concavity}
\label{alg.typeB}
\begin{algorithmic}[1]
    \Require  A function $g\colon \mathbb{Z}^n\to \mathbb{N}$ with finite support.
    \Ensure \True if $g$ is type-B log-concave extendable. \False otherwise.
    \For{$x,y\in \supp(g)$}
        \If {\textbf{not} \Call{SatisfyBExchange}{$x,y,g$}}
        \State \Return \False
        \EndIf
    \EndFor
    \State \Return \True
    \Statex
    \Procedure{SatisfyBExchange}{$x, y, g$}
    \For{$i \in [n]$}
        \If{$x_i > y_i$}
            \If{$g(x - e_i) \cdot g(y + e_i) \geq g(x) \cdot g(y)$}
                \State \textbf{continue} \Comment{Condition (1) satisfied}
            \EndIf
            \State $\text{found} \gets \False$
            \For{$j \in [n]\setminus\{i\}$}
                \If{$x_j < y_j$}
                    \If{$g(x - e_i + e_j) \cdot g(y + e_i - e_j) \geq g(x) \cdot g(y)$}
                        \State $\text{found} \gets \True$
                        \State \textbf{break} \Comment{Condition (2) satisfied}
                    \EndIf
                \ElsIf{$x_j > y_j$}
                    \If{$g(x - e_i - e_j) \cdot g(y + e_i + e_j) \geq g(x) \cdot g(y)$}
                        \State $\text{found} \gets \True$
                        \State \textbf{break} \Comment{Condition (3) satisfied}
                    \EndIf
                \EndIf
            \EndFor
            \If{\textbf{not} $\text{found}$}
                \State \Return \False
            \EndIf
        \EndIf
    \EndFor
    \State \Return \True
\EndProcedure
\end{algorithmic}
\end{algorithm}

\subsection{Computing slopes of edges in the subdivision}
\label{sec:slopecompute}
As we have seen earlier, for B-concave functions, the edges in the regular subdivision
they induce have slopes $0,\pm 1,$ or $\infty$.  Even in the cases when B-concavity
fails, we can use the failure of the exchange axioms to rule out pairs of points that
cannot form an edge in the subdivision.   

Consider two distinct points $x,y \in \dom(f)$ such that the line segment between $x$
and $y$ does {\em not} have slope $0,1,$ or $\infty$.
If one of the three exchange conditions in Definition~\ref{def.typeB} is satisfied,
then the pair $x,y$ is overshadowed by another pair.  That is, 
the midpoint of the lifted segment $[\widetilde{x},\widetilde{y}]$ lies (weakly) below
the midpoint of another lifted line segment, and the two segments are not parallel by
the slope assumption, so the the segment $[\widetilde{x},\widetilde{y}]$ cannot form an
edge in the upper hull of lifted points (although it may still lie in the interior of a
two dimensional face if the exchange inequality is satisfied as equality).  

Thus, in addition to the pairs with slope $0,1,$ or $\infty$, the only other pairs
$x,y\in \dom(f)$ that can form edges in the subdivision are the ones that fail the
exchange condition, i.e.\ the one where the function  S{\tiny ATISFY}BE{\tiny XCHANGE}
returns False in Algorithm~\ref{alg.typeB}.
This observation is used in the computation discussed in \S\ref{sec:slopes}.  Note,
however, that failure of the exchange axiom does not guarentee that the pair forms an 
edge in the subdivision, as the lifted edge can be overshadowed by some other pair that 
we have not checked.  Rigorous certification of faces of the subdivision will be 
discussed in Section~\ref{sec:facets}.


\section{Certifying log-concavity with exact computations}
\label{sec.log}

As seen in \S\ref{sec:logBcheck}, type-B log concavity can be checked easily, 
and we found that majority of knots satisfy type-B log concavity. 
On the other hand,
there are non-type-B knots, as was mentioned in \Cref{thm.1}. The aim of this section 
is to show how to rigorously check whether the coefficients of $\LG(-t_1,-t_2)$  are 
log-concave for those remaining knots. A natural idea is to compute the convex hull 
of the lifted points $\widetilde{\dom(f)}$ and check that every point is lifted to 
the upper hull.  However, due to the irrationality of logarithm values, the usual 
convex hull algorithms would not give rigorous results. In this section we provide 
an algorithm that uses exact computations only.

Let $A\subset \Z^n$ be a finite set.  An {\em almost empty simplex} in $A$ is a subset 
of the form $T=\{a_0,\dots, a_k,b\}\subset A$ such that $a_0,\dots,a_k$ are affinely 
independent, $\conv(a_0,\dots,a_k) \cap A = T$, and $b$ is in the relative interior 
of $\conv(a_0,\dots,a_k)$.  

\begin{lemma}[Lemma 3.3 of \cite{BRSVY}]
Let $A$ be a finite subset of $\Z^n$.
A function $f\colon A \rightarrow \R\cup\{-\infty\}$ is concave extendable if and only 
if for every almost empty simplex $\{a_0,\dots, a_k,b\}\subset A$ with $b = \sum_{i=0}
^k\lambda_i a_i$ with $\lambda_i > 0$ for all $i=0,\dots,k$, we have $f(b) \geq 
\sum_{i=0}^k\lambda_i f(a_i)$.
\end{lemma}

Hence to decide if a function $g\colon \mathbb{Z}^n \rightarrow \N$ with $\supp(g) = A$ 
is log-concave extendable, 
we need to check that for all almost empty simplices $\{a_0,\dots,a_k,b\}$ in $A$,
$$\log g(b) \geq \sum_{i=0}^k\lambda_i \log g(a_i),\text{ which is equivalent to }
g(b) \geq \prod_{i=0}^k g(a_i)^{\lambda_i}.
$$
Since $\lambda_i$ are rational numbers, by raising both sides to an appropriate power 
we can clear the denominators and turn this into an inequality between products of 
integers.

In our setting of $n=2$, to certify concavity we need to check all the one and two 
dimensional almost empty simplices in $A$, which are line segments and triangles.  The 
almost empty line segments have the form $\{a_0,a_1,b\}$ where $b = (a_0+a_1)/2$ and 
there are no lattice points between $a_0$ and $b$ (equivalently $a_1$ and $b$); in 
other words, coordinates of $a_0-a_1$ have gcd~$=2$.  They can be enumerated by 
checking these conditions for all pairs of points $a_0,a_1\in A$, and each such segment 
gives an inequality $g(b)^2 \geq g(a_0)g(a_1)$ to be checked.

By Pick's Theorem, the two dimensional almost empty triangles have area $\frac{3}{2}$, 
and they are equivalent under the $\GL_2(\Z)$ action to the {\em Motzkin configuration} 
$\{(0,0),(1,1),(1,2),(2,1)\}$ as depicted in~\eqref{motzkin}.  In other words, they are 
of the form $\{a_0,a_1,a_2,b\}\subseteq A$ where $b = (a_0+a_1+a_2)/3$ and the $2\times 
2$ matrix with columns $a_0-a_1$ and $a_0-a_2$ has determinant $\pm 3$.  The almost 
empty triangles can be enumerated by checking these conditions for all triples of 
points $a_0,a_1,a_2\in A$, and
each such triangle gives an inequality $g(b)^3 \geq g(a_0)g(a_1)g(a_2)$ to be checked.

For $A \subset \Z^2$, this gives us a cubic polynomial time algorithm, written in 
Algorithm~\ref{alg.exact-log-concave}, for deciding if a given function $g\colon A\to 
\mathbb{N}$ is log-concave extendable.  In line 2 of the algorithm, the condition $
\gcd(a_0-a_1)=2$ means that the line segment $[a_0,a_1]$ has exactly one inerior 
lattice point $(a_0+a_1)/2$.
In line 5, the condition $\gcd(a_0-a_1)=\gcd(a_0-a_2)=\gcd(a_1-a_2)=1$ says that for 
each of the vectors $a_0-a_1, a_0-a_2$ and $a_1-a_2$, the gcd of the coordinates is 
one.  This is equivalent to saying there are no interior lattice points along each of 
the three edges of the triangle.  This condition and the determinant condition for the 
triangle to have area $\frac{3}{2}$ characterize almost empty triangles and ensure that 
$(a_0+a_1+a_2)/3$ is a lattice point.

\begin{algorithm}[h] 
\caption{Certify log-concavity with exact computation}
\label{alg.exact-log-concave}
\begin{algorithmic}[1]
    \Require A function $g\colon\mathbb{Z}^2\to \mathbb{N}$ with finite support.
    \Ensure \True if $g$ is log-concave extendable. \False otherwise.
    \For{$a_0,a_1\in \supp(g)$}
        \If{$\gcd(a_0-a_1) = 2$ and $g((a_0+a_1)/2)^2< g(a_0)\cdot g(a_1)$}
            \State \Return \False
        \EndIf
    \EndFor
    \For{$a_0,a_1,a_2 \in \supp(g) $}
        \If {$\gcd(a_0-a_1)=\gcd(a_0-a_2)=\gcd(a_1-a_2)=1$ and  $|\det(a_0-a_1,a_0-a_2)| = 3$} 
            \If{$g((a_0+a_1+a_2)/3)^3 < g(a_0)\cdot g(a_1)\cdot  g(a_2)$}
                \State \Return \False
            \EndIf
        \EndIf
    \EndFor
    \State \Return \True
\end{algorithmic}
\end{algorithm}


\section{Certifying facets of the subdivision}
\label{sec:facets}

After we have checked that the LG coefficients are log-concave, we would like to study 
the structure of the subdivision induced by the log coefficients, that is, the 
projection of the upper hull of the lifted points.  When we have type-B log concavity, we know that the edges of the subdivision have slopes $0,\pm 1$, or $\infty.$  For the rest of the instances, we would like to find out which other slopes can appear, as discussed in \S\ref{sec:slopes}.  
We saw in \S\ref{sec:slopecompute} how to rule out, using the exchange conditions, many pairs that do not form an edge. 
Here we show how to certify all the faces that do appear in the subvision.

We can compute the upper hull using standard convex hull algorithms, but lifted points have irrational logarithm values as coordinates, so the output may reflect round off errors. 
In this section we discuss how we can rigorously verify that a subdivision computed using numerical estimates is indeed correct.  We focus on facets (maximal dimensional faces) only for simplicity, as correctness facets imply correctness of all other faces.

Given a subset $S \subseteq \dom(f)$ such that $\conv(S)$ is full-dimensional, we can check, using exact arithmetic, whether the lifted points $\widetilde S$ belong to a common upper face of the convex hull of $\widetilde{\dom(f)}$ as follows. 
By the definition of faces, the lifted set $\widetilde S$ is contained in a common upper face if and only if there exists a vector $w \in \R^{n}$ such that \begin{equation}
\label{eqn:max}
    w \cdot s + \log f(s) \geq w \cdot t + \log f(t), \end{equation} 
    for every $s\in S$ and $t \in \dom(f)$.
The vector $w$ satisfies the linear equations
\begin{equation}
\label{eqn:max-ln}
     w \cdot (s - s') = \log \frac{f(s)}{f(s')}
\end{equation}
for all pairs of distinct points $s,s' \in S$. When $\conv(S)$ is full-dimensional, \eqref{eqn:max-ln} implies that the vector $w$ in \eqref{eqn:max} is unique up to normalization, if it exists. 

Choose affinely independent elements $s_0,s_1,\dots,s_n \in S$, so that the differences $s_i-s_0$ for $i=1,\dots,n$ form a basis of $\Q^n$. Then any $w$ satisfying \eqref{eqn:max} is a solution of the linear system $C w = b$ where
\begin{equation}
    C= \left(\begin{matrix}
        s_1 -s_0\\
        s_2-s_0\\
        \cdots\\
        s_n-s_0
    \end{matrix}\right),\quad b = \left(\begin{matrix}
        \log(f(s_1)/f(s_0))\\
        \log(f(s_2)/f(s_0))\\
        \cdots\\
        \log(f(s_n)/f(s_0))
    \end{matrix}\right).
\end{equation}
Inverting the matrix $C$ gives $w = C^{-1} b$ where $C^{-1}$ is a matrix over $\Q$. Hence $w_i = \sum_{j=0}^n r_{ij} \log(f(s_j)/f(s_0))$ for some rational numbers $r_{ij}$ which are entries of the matrix $C^{-1}$.
Exponentiating gives \begin{equation}
\label{eqn:exp}
    \exp{w_i} = \prod_{j=1}^n f(s_j)^{r_{ij}}/f(s_0)^{r_{ij}}.\end{equation}
Similarly, exponentiating both sides of \eqref{eqn:max} gives
 \begin{equation}
     \label{eqn:final}
 f(s) \prod_{i=1}^n(\exp{w_i})^{s_i}  \geq f(t) \prod_{i=1}^n (\exp{w_i})^{t_i}.
  \end{equation}
which can be checked in exact arithmetic by substituting $\exp w_i$ with \eqref{eqn:exp}, since $s_i, t_i$ are integers, and we can clear the denominators in $r_{ij}$ by raising the expressions to integer powers.  If the inequalities \eqref{eqn:final} are satisfied for all $s \in S$ and $t\in \dom(f)$, then all points of $S$ belong to a common face of the subdivision. Otherwise they do not.


\subsection*{Acknowledgements}
The authors wish to thank Nathan Dunfield, June Huh, Bowen Li, Kevin A. Zhou 
for useful conversations. This material is based upon work supported by the
U.S. National Science Foundation under Grant DMS-2424139 while SYL was
in residence at the Simons Laufer Mathematical Sciences Institute in Berkeley, California, during the Spring 2026 semester. SYL was also supported by the 
US NSF grant DMS-2303572 during the same semester. JY was partially supported 
by the US NSF grant DMS-2348701.


\bibliographystyle{hamsalpha}
\bibliography{biblio}


\appendix
\clearpage
\section{Non-type-B knots}
\label{append.nontypeB}

In this appendix we give the list of 544 non-type-B alternating knots with
at most 19 crossings.

\begin{subequations}
\tiny
\begin{equation}
    \begin{aligned}
    13a4593 && 
    15a57208 && 
    15a57492 && 
    15a79064 && 
    15a79080 && 
    15a82231 && 
    15a84755 \\ 
    15a84831 && 
    16a42055 && 
    16a42757 && 
    16a116388 && 
    16a158905 && 
    16a161990 && 
    16a165144 \\ 
    16a236920 && 
    16a238636 && 
    16a239190 && 
    16a245563 && 
    16a247221 && 
    16a290681 && 
    16a310628 \\ 
    16a356928 && 
    16a375806 && 
    17ah0000557 && 
    17ah0000628 && 
    17ah0001073 && 
    17ah0001123 && 
    17ah0001726 \\ 
    17ah0001775 && 
    17ah0002195 && 
    17ah0002326 && 
    17ah0002916 && 
    17ah0002958 && 
    17ah0004297 && 
    17ah0004669 \\ 
    17ah0005356 && 
    17ah0005765 && 
    17ah0010753 && 
    17ah0013525 && 
    17ah0013526 && 
    17ah0015069 && 
    17ah0015479 \\ 
    17ah0016154 && 
    17ah0017669 && 
    17ah0017670 && 
    17ah0017671 && 
    17ah0017774 && 
    17ah0017775 && 
    17ah0018823 \\ 
    17ah0023338 && 
    17ah0023339 && 
    17ah0032520 && 
    17ah0032521 && 
    17ah0058154 && 
    17ah0058155 && 
    17ah0058156 \\ 
    17ah0067013 && 
    17ah0067014 && 
    17ah0067015 && 
    17ah0075902 && 
    17ah0075903 && 
    17ah0076994 && 
    17ah0076995 \\ 
    17ah0088980 && 
    17ah0094929 && 
    17ah0094930 && 
    17ah0098778 && 
    17ah0125817 && 
    17ah0125818 && 
    18ah0008327 \\ 
    18ah0008442 && 
    18ah0008954 && 
    18ah0009082 && 
    18ah0011197 && 
    18ah0011205 && 
    18ah0015197 && 
    18ah0020678 \\ 
    18ah0021640 && 
    18ah0026052 && 
    18ah0026270 && 
    18ah0026271 && 
    18ah0027697 && 
    18ah0029518 && 
    18ah0030228 \\ 
    18ah0031884 && 
    18ah0032547 && 
    18ah0034188 && 
    18ah0034189 && 
    18ah0034385 && 
    18ah0034386 && 
    18ah0034387 \\ 
    18ah0035601 && 
    18ah0037931 && 
    18ah0038067 && 
    18ah0038068 && 
    18ah0039866 && 
    18ah0039867 && 
    18ah0052456 \\ 
    18ah0053924 && 
    18ah0053925 && 
    18ah0055634 && 
    18ah0061629 && 
    18ah0062478 && 
    18ah0063269 && 
    18ah0063270 \\ 
    18ah0063970 && 
    18ah0063971 && 
    18ah0063972 && 
    18ah0065112 && 
    18ah0071568 && 
    18ah0072582 && 
    18ah0072728 \\ 
    18ah0081656 && 
    18ah0081759 && 
    18ah0083350 && 
    18ah0083453 && 
    18ah0083830 && 
    18ah0083831 && 
    18ah0083832 \\ 
    18ah0084373 && 
    18ah0084374 && 
    18ah0084375 && 
    18ah0084724 && 
    18ah0085776 && 
    18ah0089358 && 
    18ah0105300 \\ 
    18ah0106228 && 
    18ah0106229 && 
    18ah0106230 && 
    18ah0114407 && 
    18ah0114408 && 
    18ah0114409 && 
    18ah0117721 \\ 
    18ah0117722 && 
    18ah0117723 && 
    18ah0120816 && 
    18ah0124785 && 
    18ah0129072 && 
    18ah0129073 && 
    18ah0129074 \\ 
    18ah0129311 && 
    18ah0129312 && 
    18ah0134690 && 
    18ah0134691 && 
    18ah0134692 && 
    18ah0135051 && 
    18ah0155329 \\ 
    18ah0155330 && 
    18ah0155331 && 
    18ah0168740 && 
    18ah0168741 && 
    18ah0173125 && 
    18ah0173126 && 
    18ah0173460 \\ 
    18ah0173461 && 
    18ah0173794 && 
    18ah0173795 && 
    18ah0173796 && 
    18ah0175353 && 
    18ah0175354 && 
    18ah0179439 \\ 
    18ah0179440 && 
    18ah0193698 && 
    18ah0193699 && 
    18ah0194275 && 
    18ah0194276 && 
    18ah0201520 && 
    18ah0201521 \\ 
    18ah0209342 && 
    18ah0209343 && 
    18ah0210247 && 
    18ah0210248 && 
    18ah0224113 && 
    18ah0224114 && 
    18ah0254539 \\ 
    18ah0254540 && 
    18ah0263728 && 
    18ah0265647 && 
    18ah0265648 && 
    18ah0275603 && 
    18ah0275604 && 
    18ah0312457 \\ 
    18ah0312458 && 
    18ah0344472 && 
    18ah0376258 && 
    18ah0376259 && 
    18ah0376260 && 
    18ah0405731 && 
    18ah0433912 \\ 
    18ah0433913 && 
    18ah0462101 && 
    18ah0462102 && 
    18ah0462103 && 
    18ah0485631 && 
    18ah0485632 && 
    18ah0485633 \\ 
    18ah0557564 && 
    18ah0557565 && 
    18ah0567702 && 
    18ah0675713 && 
    18ah0675714 && 
    18ah0722353 && 
    18ah0722354 \\ 
    18ah0726516 && 
    18ah0839509 && 
    18ah0839510 && 
    18ah0839511 && 
    18ah0840532 && 
    18ah1082631 && 
    18ah1082632 \\ 
    18ah1082633 && 
    19ah00000670 && 
    19ah00000823 && 
    19ah00001428 && 
    19ah00001533 && 
    19ah00001675 && 
    19ah00001782 \\ 
    19ah00002348 && 
    19ah00002490 && 
    19ah00002794 && 
    19ah00003274 && 
    19ah00003324 && 
    19ah00003433 && 
    19ah00003597 \\ 
    19ah00003778 && 
    19ah00004471 && 
    19ah00004731 && 
    19ah00004895 && 
    19ah00005063 && 
    19ah00005153 && 
    19ah00006289 \\ 
    19ah00006572 && 
    19ah00007359 && 
    19ah00007449 && 
    19ah00007577 && 
    19ah00009090 && 
    19ah00009205 && 
    19ah00009310 \\ 
    19ah00009986 && 
    19ah00011484 && 
    19ah00011617 && 
    19ah00012257 && 
    19ah00015976 && 
    19ah00016059 && 
    19ah00018522 \\ 
    19ah00021147 && 
    19ah00021148 && 
    19ah00023168 && 
    19ah00025954 && 
    19ah00027074 && 
    19ah00027277 && 
    19ah00028136 \\ 
    19ah00028152 && 
    19ah00029003 && 
    19ah00030079 && 
    19ah00030080 && 
    19ah00030081 && 
    19ah00031045 && 
    19ah00031046 \\ 
    19ah00033003 && 
    19ah00034136 && 
    19ah00035679 && 
    19ah00035680 && 
    19ah00035681 && 
    19ah00036158 && 
    19ah00036422 \\ 
    19ah00036423 && 
    19ah00043932 && 
    19ah00046320 && 
    19ah00046321 && 
    19ah00047591 && 
    19ah00047592 && 
    19ah00048519 \\ 
    19ah00048520 && 
    19ah00048521 && 
    19ah00051814 && 
    19ah00056023 && 
    19ah00056280 && 
    19ah00056281 && 
    19ah00057110 
    \end{aligned}
\end{equation}
\begin{equation}
    \begin{aligned}
    19ah00060709 && 
    19ah00060710 && 
    19ah00061229 && 
    19ah00061230 && 
    19ah00061231 && 
    19ah00061729 && 
    19ah00068122 \\ 
    19ah00068123 && 
    19ah00069827 && 
    19ah00069828 && 
    19ah00074100 && 
    19ah00078025 && 
    19ah00089752 && 
    19ah00091273 \\ 
    19ah00092520 && 
    19ah00092521 && 
    19ah00092522 && 
    19ah00095519 && 
    19ah00095520 && 
    19ah00100714 && 
    19ah00112892 \\ 
    19ah00116179 && 
    19ah00116180 && 
    19ah00125509 && 
    19ah00125510 && 
    19ah00125511 && 
    19ah00149036 && 
    19ah00149037 \\ 
    19ah00149038 && 
    19ah00154779 && 
    19ah00164607 && 
    19ah00164608 && 
    19ah00164609 && 
    19ah00165419 && 
    19ah00165420 \\ 
    19ah00165421 && 
    19ah00165979 && 
    19ah00165980 && 
    19ah00167522 && 
    19ah00167523 && 
    19ah00167524 && 
    19ah00167995 \\ 
    19ah00167996 && 
    19ah00172745 && 
    19ah00172746 && 
    19ah00176480 && 
    19ah00176481 && 
    19ah00176482 && 
    19ah00197685 \\ 
    19ah00197686 && 
    19ah00197687 && 
    19ah00198436 && 
    19ah00198437 && 
    19ah00198438 && 
    19ah00199023 && 
    19ah00199024 \\ 
    19ah00199025 && 
    19ah00200522 && 
    19ah00204066 && 
    19ah00204067 && 
    19ah00216085 && 
    19ah00218536 && 
    19ah00229181 \\ 
    19ah00229182 && 
    19ah00229183 && 
    19ah00229694 && 
    19ah00229695 && 
    19ah00229696 && 
    19ah00240251 && 
    19ah00240252 \\ 
    19ah00240253 && 
    19ah00247366 && 
    19ah00247367 && 
    19ah00247368 && 
    19ah00249141 && 
    19ah00255800 && 
    19ah00255801 \\ 
    19ah00267603 && 
    19ah00267604 && 
    19ah00269876 && 
    19ah00269877 && 
    19ah00287530 && 
    19ah00287531 && 
    19ah00290912 \\ 
    19ah00290913 && 
    19ah00290914 && 
    19ah00291432 && 
    19ah00291433 && 
    19ah00311546 && 
    19ah00311547 && 
    19ah00311548 \\ 
    19ah00318817 && 
    19ah00318818 && 
    19ah00327749 && 
    19ah00327750 && 
    19ah00327751 && 
    19ah00342677 && 
    19ah00342678 \\ 
    19ah00348626 && 
    19ah00348627 && 
    19ah00353568 && 
    19ah00374875 && 
    19ah00374876 && 
    19ah00374877 && 
    19ah00380955 \\ 
    19ah00380956 && 
    19ah00388152 && 
    19ah00388153 && 
    19ah00388154 && 
    19ah00396075 && 
    19ah00396076 && 
    19ah00417243 \\ 
    19ah00417244 && 
    19ah00430342 && 
    19ah00430343 && 
    19ah00434295 && 
    19ah00434296 && 
    19ah00442518 && 
    19ah00442519 \\ 
    19ah00456229 && 
    19ah00457985 && 
    19ah00457986 && 
    19ah00468124 && 
    19ah00474963 && 
    19ah00484585 && 
    19ah00484586 \\ 
    19ah00507186 && 
    19ah00507187 && 
    19ah00508105 && 
    19ah00512755 && 
    19ah00513664 && 
    19ah00533220 && 
    19ah00541200 \\ 
    19ah00543739 && 
    19ah00543740 && 
    19ah00548358 && 
    19ah00548359 && 
    19ah00570008 && 
    19ah00589919 && 
    19ah00595382 \\ 
    19ah00595383 && 
    19ah00602292 && 
    19ah00675471 && 
    19ah00679880 && 
    19ah00688803 && 
    19ah00699814 && 
    19ah00708808 \\ 
    19ah00708809 && 
    19ah00744225 && 
    19ah00750807 && 
    19ah00750808 && 
    19ah00791710 && 
    19ah00792449 && 
    19ah00826502 \\ 
    19ah00826503 && 
    19ah00826504 && 
    19ah00826505 && 
    19ah00861338 && 
    19ah00862502 && 
    19ah00870541 && 
    19ah00870542 \\ 
    19ah00870543 && 
    19ah00872067 && 
    19ah00873479 && 
    19ah00876877 && 
    19ah00885123 && 
    19ah00885124 && 
    19ah00889421 \\ 
    19ah00889422 && 
    19ah00950595 && 
    19ah00950596 && 
    19ah00950597 && 
    19ah00956313 && 
    19ah00956314 && 
    19ah01009865 \\ 
    19ah01009866 && 
    19ah01017186 && 
    19ah01017187 && 
    19ah01017188 && 
    19ah01032764 && 
    19ah01047030 && 
    19ah01052576 \\ 
    19ah01052577 && 
    19ah01052578 && 
    19ah01058450 && 
    19ah01073354 && 
    19ah01073355 && 
    19ah01136558 && 
    19ah01136997 \\ 
    19ah01136998 && 
    19ah01174448 && 
    19ah01176771 && 
    19ah01176772 && 
    19ah01176773 && 
    19ah01206730 && 
    19ah01220595 \\ 
    19ah01220596 && 
    19ah01239224 && 
    19ah01302651 && 
    19ah01324708 && 
    19ah01324709 && 
    19ah01331557 && 
    19ah01331558 \\ 
    19ah01331559 && 
    19ah01331560 && 
    19ah01331561 && 
    19ah01331562 && 
    19ah01336247 && 
    19ah01441990 && 
    19ah01441991 \\ 
    19ah01441992 && 
    19ah01441993 && 
    19ah01441994 && 
    19ah01441995 && 
    19ah01566066 && 
    19ah01645570 && 
    19ah01645571 \\ 
    19ah01645572 && 
    19ah01645573 && 
    19ah01662198 && 
    19ah01662199 && 
    19ah01881434 && 
    19ah01881435 && 
    19ah02185587 \\ 
    19ah02185588 && 
    19ah02185589 && 
    19ah02215962 && 
    19ah02215963 && 
    19ah02251155 && 
    19ah02269316 && 
    19ah02269317 \\ 
    19ah02269318 && 
    19ah02338431 && 
    19ah02338432 && 
    19ah02460871 && 
    19ah02460872 && 
    19ah02460873 && 
    19ah02460874 \\ 
    19ah02460875 && 
    19ah02460876 && 
    19ah02475075 && 
    19ah02475076 && 
    19ah02475077 && 
    19ah02806141 && 
    19ah02806142 \\ 
    19ah03077163 && 
    19ah04258109 && 
    19ah04258110 && 
    19ah04258111 && 
    19ah04344793 && 
    19ah04344794 && 
    19ah04344795 \\ 
    19ah05354678 && 
    19ah05354679 && 
    19ah05354680 && 
    19ah05478109 && 
    19ah05478110 && 
    19ah05478111 && 
    19ah05478112 \\ 
    19ah05478113 && 
    19ah05478114 && 
    19ah05705309 && 
    19ah05826316 && 
    19ah06038586 && 
    \end{aligned}
\end{equation}
\normalsize
\end{subequations}

\end{document}